\documentclass[11pt]{amsart}
\usepackage{stmaryrd}
\usepackage{amssymb}
\usepackage{csvsimple}
\usepackage{siunitx}
\usepackage{hyperref}
\usepackage{url}
\usepackage{graphicx,color}
\usepackage{lineno}           
\usepackage[style=trad-abbrv,url=false,doi=true,
  eprint=true,isbn=false,backend=biber]{biblatex}

\graphicspath{{pics/},{figs/}}
\def\R{\mathbb{R}}

\def\Rinf{\R\cup \{+\infty\}}
\def\cA{\mathcal{A}}

\def\cF{\mathcal{F}}

\def\cI{\mathcal{I}}

\def\cO{\mathcal{O}}

\def\cV{\mathcal{V}}

\def\a{\alpha}

\def\k{\kappa}
\def\l{\lambda}

\def\veps{\varepsilon}

\def\wto{\rightharpoonup}

\newcommand{\dv}[1]{\,{\mathrm d}#1}

\newcommand{\wcheck}[1]{#1\hspace{-.8ex}\mbox{\huge {\lower.45ex \hbox{$\textstyle \check{}$}}} \hspace{.5ex}}

\let\oldmarginpar\marginpar
\renewcommand\marginpar[1]{
  \oldmarginpar[\raggedleft\footnotesize #1]
  {\raggedright\footnotesize #1}}

\newtheorem{definition}{Definition}

\newtheorem{proposition}[definition]{Proposition}

\newtheorem{corollary}[definition]{Corollary}
\newtheorem{remark}[definition]{Remark}

\numberwithin{definition}{section}
\definecolor{modmag}{RGB}{179,0,229}

\renewcommand{\text}{\textnormal}

\newtheorem{conjecture}[definition]{Conjecture}
\def\knotevolve{\textsc{Knotevolve}\ }

\begin{document}
\title{Computing confined elasticae}
\author[S. Bartels]{S\"oren Bartels}
\author[P. Weyer]{Pascal Weyer}
\email{bartels@mathematik.uni-freiburg.de}
\address{Abteilung f\"ur Angewandte Mathematik,  
Albert-Ludwigs-Universit\"at Freiburg, Hermann-Herder-Str.~10, 
79104 Freiburg i.~Br., Germany}
\date{\today}
\renewcommand{\subjclassname}{
\textup{2010} Mathematics Subject Classification}
\subjclass[2010]{65N30 (35Q74 65N12 74K10)}
\begin{abstract}
A numerical scheme for computing arc-length parametrized curves 
of low bending energy that are confined to convex domains 
is devised. The convergence of the discrete formulations to a continuous
model and the unconditional stability of an iterative scheme are addressed.
Numerical simulations confirm the theoretical results and lead to a 
classification of observed optimal curves within spheres. 
\end{abstract}
\keywords{rods, elasticity, constraints, numerical scheme}

\maketitle

\section{Introduction} 
Equilibrium configurations of thin elastic rods have been of interest since the times of Euler. 
The mathematical modelling of these deformable structures has been reduced 
from three dimensions to a one-dimensional problem for the center-line of the rod $u:I\to\R^3$,
cf.~\cite{Kirchhoff1859,Green1967,Dill1992,Mora2003,Bartels2019}. In the bending regime, 
the rod is inextensible so that $|u'|=1$ holds on $I$. Considering a circular cross section
and omitting twist contributions, the elastic energy reduces to the functional 
\begin{equation*}
E_\text{bend}[u]=\frac\kappa2\int_I|u''(x)|^2\dv{x}
\end{equation*}
for a parameter $\kappa>0$ that describes the bending rigidity.  Elasticae, i.e. rods of minimal 
bending energy, can be stated explicitly e.g. for periodic boundary conditions 
\cite{Langer1984,Langer1985}. Applications of elastic thin rods include DNA modelling 
\cite{Shi1994,Furrer2000,Balaeff2006}, the movement of actin filaments in cells 
\cite{Manhart2015} or of thin microswimmers \cite{Ranner2020}, the fabrication 
of textiles \cite{Grothaus2016}, and investigating the reach of a rod injected into a cylinder \cite{Reis16}. 

To obtain minimally bent elastic rods, the bending energy can be reduced by a gradient-flow 
approach. This method can be used for analytic considerations, cf. for instance 
\cite{Langer1985,Dziuk2002,Oelz2011,DallAcqua2014,Reiter21} and numerical computations 
\cite{Deckelnick2005,Barrett2008,Deckelnick2009,Barrett2010,Barrett2011,Bartels2013,Walker16,
Bartels2018,Barrett2019,BGN_JCP_19,Doerfler2019,Bonito22}. An efficient finite-element approach with an accurate 
treatment of the inextensibility condition
can be used to find equilibria of free elastic rods \cite{Bartels2013} and 
self-avoiding rods \cite{BarRei21,Bartels2018}. It can also be generalized to include 
twist contributions defined via torsion quantities \cite{BarRei20}. We follow 
common conventions and refer to rods as {\em elastic curves} when twist contributions 
are omitted. 

In this manuscript, a generalization of the existing scheme to calculate elasticae 
of \emph{confined} elastic curves is proposed. Confinements of elastic structures 
arise on a variety of length scales, such as DNA plasmids or biopolymers inside a 
cell or chamber \cite{Choi2005,Ostermeir2010}. The boundary of closely packed elastic 
sheets or a wire in a container can be modelled as confined elastic rods in two dimensions 
\cite{Donato2003,Boue2006}. A planar setting has been assessed in terms of 
phase-field modelling \cite{Dondl2011,Wojtowytsch2018}; a numerical scheme for thick
elastic curves in containers is devised in \cite{Walker2021}.
	
We propose an approach that can be used for rods embedded in arbitrary dimensions confined 
to convex domains. For the mathematical modelling, we use a gradient flow to minimize the 
bending energy. The admissible rod configurations during the flow are restricted to a domain 
$D\subset\R^3$.

The task to unbend a rod inside $D$ can be translated to minimizing $E_\text{bend}$ among all 
\[
u\in \cA\cap D, \quad \mathcal{A}=\{v\in H^2(I;\R^3)\,: |v'|=1\text{ a.e.},\, L_\text{bc}(v)=\ell_\text{bc} \}.
\]
The bounded linear operator $L_\text{bc}:H^2(I;\R^3) \to \R^\ell$ realizes appropriate boundary 
conditions. We restrict our considerations to those subsets $D$ that can be written as
finite intersections of \emph{simple quadratic confinements} $D_r$, $r=1,2,\dots,R$, i.e., 
\[
D= \cap_{r=1}^R D_r, \quad D_r = \{y\in\R^3: |y|_{D_r}^2 = y \cdot G_{D_r}y\leq 1\}
\]
for symmetric positive semi-definite matrices $G_{D_r} \in \R^{3\times 3}$. We call the finite
intersection a \emph{composite quadratic confinement}.
For ease of presentation we often consider one set $D_r$ and then omit the index $r$.  
Some basic simple quadratic confinements are the ball with radius $R$ and $G_D=I_3/R^2$, 
the ellipsoid with radii $R_1,R_2,R_3$ and $(G_D)_{ij}=\delta_{ij}/R_i^2$, or the space between 
two parallel planes with distance $2R$ with normal vector $n$ and $G_D=nn^t/R^2$. Boxes and finite 
cylinders can be constructed as composite quadratic confinements. In general, any simple or 
composite quadratic confinement is a convex, closed, and connected set. 

We enforce the confinement via a potential approach, so a non-negative term is added to 
the bending energy whenever the curve violates the confining restrictions. We define 
a potential $V_D:\R^3\to\R$ for a simple quadratic confinement $D$ that vanishes in $D$ 
and is strictly positive on $\R^3\backslash D$ via
\begin{equation*}
V_D(y)=\frac12\left(|y|_D-1\right)_+^2= \frac12 |y|_D^2 + \frac12 V_D^\text{cv}(y),
\end{equation*}
where the concave part $V_D^\text{cv}$ is given by the continuous function
\begin{equation*}
V_D^\text{cv}(y)=\begin{cases}
-|y|_D^2,&\text{ if $y\in D$,}\\
-2|y|_D+1&\text{ else.}
\end{cases}.
\end{equation*}
The potential is used to define a penalizing \emph{confinement energy} functional
\begin{equation*}
E_D[u]=\int_I|u(x)|_D^2+V_D^\text{cv}(u(x))\dv{x},
\end{equation*}
which is by the definition of the potential non-negative and zero if and only if 
the curve entirely lies within $D$. 
For a composite confinement defined via a family $(D_r)_{r=1,\dots,R}$ of simple 
quadratic confinements we sum the corresponding confinement energies up, i.e.,
\begin{equation}
E_D[u]=\sum_{r=1}^R E_{D_r}[u],\quad V_D(y)=\sum_{r=1}^R V_{D_r}(y).\label{eq:composite-energies}
\end{equation}
We remark that translated domains and half-spaces, e.g., $D = \{y\in \R^3: |y-y_D|_D^2 \le 1\}$ and
$D = \{y\in \R^3: a_D\cdot y \le 1 \}$ can be similarly treated. 

Given $\veps>0$, a curve $u_\varepsilon\in \mathcal{A}$ is called a 
\emph{(approximately) confined elastica} if it is stationary for the functional
\[
E_\varepsilon[u]= E_\text{bend}[u] +\frac1{2\varepsilon}E_D[u]
\]
in the set $\mathcal{A}$. If $V_D(u_\varepsilon)=0$ almost everywhere on $I$, 
the rod is called \emph{exactly confined elastica.} The parameter $\varepsilon$ 
determines the steepness of the quadratic well potential and defines a length-scale 
for the penetration depth of the curve into the space outside of $D$.

Considering a simple quadratic confinement $D\subset\R^3$, we let  $V_D\in C^1(\R^3;\R)$ 
be the corresponding quadratic-well potential, and choose $\veps>0$. Trajectories 
$u\in H^1([0,T];L^2(I;\R^3))\cap L^\infty([0,T];\mathcal{A})$ 
are defined by gradient flow evolutions. In particular, for an inner product $(\cdot,\cdot)_\star$ 
on $L^2(I;\R^3)$ and an initial configuration $u(0,x)=u_0(x)$, we define the temporal 
evolution as the solution of the time-dependent nonlinear system of partial differential 
equations
\begin{equation}
\begin{split}
 (\partial_t u,v)_\star+\kappa(u'',v'')& +\varepsilon^{-1}(u,G_Dv) \\
&   = - (2\varepsilon)^{-1}(\nabla V_D^\text{cv}(u),v) -(\lambda u',v') \label{eq:gradient-flow}
\end{split}
\end{equation}
for test functions $v\in \mathcal{V}$ with a suitable set $\cV$ and all $t\in [0,T]$. 
The function $\lambda\in L^1([0,T]\times I)$ is a Lagrange multiplier associated with the
arc-length condition. Confined elasticae are stationary points for~\eqref{eq:gradient-flow}.

For time discretization, we use backward differential quotients. Let $\tau>0$ be the 
fixed time-step and let $k\geq0$ be a non-negative integer. We set $u^0:=u_0$ and 
define the time-step
\begin{equation*}
d_t u^{k+1}=\frac{u^{k+1}-u^k}{\tau}.
\end{equation*}
The gradient flow system is evaluated implicitly except for the concave confinement energy, 
which is handled explicitly due to its non-linearity and anti-monotonicity, and the Lagrange
multiplier term, which is treated semi-implicitly. We hence have
\begin{equation}\begin{split}
(d_tu^{k+1},v)_\star+\kappa([u^{k+1}]'',v'')& + \varepsilon^{-1}(u^{k+1},G_Dv)\\ 
& =  - (2\varepsilon)^{-1}(\nabla V_D^\text{cv}(u^k),v)- (\lambda^{k+1}[u^k]',v') \label{eq:discr_grad_flow}
\end{split}\end{equation}
for suitable test curves $v\in \mathcal{V}$. To ensure that the parametrization by arc-length 
is approximately preserved throughout the gradient flow, the constraint $|[u^k]'|^2 =1$ is linearized. 
This yields the first order orthogonality condition
\begin{equation}
[u^k]'\cdot[d_tu^{k+1}]'=0\text{ on $I$.}\label{eq:ortho-semidiscr}
\end{equation}
By imposing the same condition on test curves, i.e., 
\begin{equation}
[u^k]'\cdot v'=0 \text{ on $I$,}\label{eq:ortho-semidiscr-test}
\end{equation}
the Lagrange multiplier term disappears in~\eqref{eq:ortho-semidiscr}.
Given $u^0,u^1,\ldots,u^k\in H^2(I;\R^3)$, there are unique functions $d_tu^{k+1}\in H^2(I;\R^3)$ 
that solve the gradient flow equation~\eqref{eq:discr_grad_flow} with all $v$
satisfying~\eqref{eq:ortho-semidiscr-test} and $L_\text{bc}[v]=0$.
This is a direct consequence of the Lax-Milgram lemma.

For numerical computations, we subdivide $I$ into a partition $\mathcal{P}_h$ of maximal 
length $h$, which can be represented by the nodes $x_0<x_1<\ldots<x_N$. We use the space of 
piecewise cubic, globally continuously differentiable splines on $\mathcal{P}_h$ as a 
conforming subspace $V_h\subset H^2(I)$.  On an interval $[x_i,x_{i+1}]$, these functions 
are entirely defined by the values and the derivatives at the endpoints. We also employ
the space of piecewise linear, globally continuous finite element functions that are
determined by the nodal values and denote the set by $W_h$. The corresponding interpolation
operators are denoted as $\mathcal{I}_{3,h}$ and $\mathcal{I}_{1,h}$, respectively. We impose 
the orthogonality of $d_tu_h^{k+1}$ and $u_h^k$ only at the nodes. The confinement
quantities are evaluated by mass lumping, so only the values at the nodes are required
\begin{equation*}
(v,w)_h:=\int_I\mathcal{I}_{1,h}(v\cdot w)\dv{x}.
\end{equation*}
In the nodal points, the concavity of $V_D^\text{cv}$ is utilized to prove an energy 
monotonicity property.

The discrete admissible set is defined via
\[
\cA_h := \{u_h \in V_h^3: |u_h'(x_i)|^2 =1,\, i=0,1,\dots,N, \ L_\text{bc}[u_h] = \ell_\text{bc}\},
\]
and we write $u_h\in \cA_h \cap D$ if $u_h(x_i)\in D$ for $i=0,1,\dots,N$. The set
of test functions relative to $u_h$ is 
\[
\cF_h[u_h] := \{ v_h \in V_h^3: u_h'(x_i) \cdot v_h'(x_i) = 0, \, i=0,1,\dots,N, \ L_\text{bc}[v_h] = 0 \}.
\]
We thus obtain the following fully practical numerical scheme to compute
confined elasticae: Given $u_h^0 \in \cA_h$ define $u_h^1,\ldots,u_h^k\in V_h^3$ by calculating
$d_tu_h^{k+1}\in \cF_h[u_h^k]$ such that
\begin{equation}\begin{split}
(d_tu_h^{k+1},v_h)_\star+\kappa([u_h^k]''+\tau [d_tu_h^{k+1}]'',v_h'')
& +\varepsilon^{-1}(u_h^k+\tau d_tu_h^{k+1},G_Dv_h)_h \\
& = - (2\varepsilon)^{-1}(\nabla V_D^\text{cv}(u_h^k),v_h)_h \label{eq:fully_discr_grad_flow}
\end{split}\end{equation}
for all $v_h\in \cF_h[u_h^k]$.
	
The remainder of this paper is structured into a first part proving the convergence of 
the proposed numerical scheme and into a second part that presents results of numerical 
experiments and describes confined elasticae for closed rods in balls.
The numerical simulations were done in the web application \knotevolve \cite{BaFaWe20}
which is accessible at \url{aam.uni-freiburg.de/knotevolve}.

\section{Convergence results}
In this section, we provide convergence results following ideas from~\cite{BarPal20-pre}.
The first result establishes the unconditional variational convergence of the discrete minimization 
problems to the continuous one defining confined elasticae.
The following partial $\Gamma$ convergence result relies on a regularity condition and
is a consequence of conformity properties of the discrete model. 

\begin{proposition}
Define $E_h: H^2(I;\R^3) \to \Rinf$ via 
\[
E_{h,\veps}[u_h] = \frac{\k}{2} \int_I |u_h''|^2 \dv{x} + \frac{1}{2\veps} \int_I \cI_{1,h} V_D(u_h) \dv{x}
\]
if $u_h \in \cA_h$ and $E_h[u_h] = + \infty$ if $u_h\in H^2(I;\R^3)\setminus \cA_h$. Analogously,
let
\[
E_\text{bend}[u] = \frac{\k}{2} \int_I |u''|^2 \dv{x} 
\]
for $u\in \cA\cap D$ and $E_\text{bend}[u]=+\infty$ if $u\in H^2(I;\R^3) \setminus \cA$. \\
(i) For every sequence $(u_h)_{h>0} \subset H^2(I;\R^3)$ with weak limit $u \in H^2(I)$
we have $E_\text{bend}[u] \le \liminf_{(h,\veps)\to 0} E_{h,\veps}[u_h]$. \\
(ii) For every $u\in \cA\cap D$ with $u \in H^3(I;\R^3)$ there exists a sequence
$(u_h)_{h>0} \subset H^2(I;\R^3)$ such that
$\lim_{(h,\veps)\to 0} E_{h,\veps}[u_h] = E_\text{bend}[u]$.
\end{proposition}

\begin{proof}
Throughout this proof we write $h\to 0$ for a sequence $(h,\veps)\to 0$. \\
(i) We consider a sequence $(u_h)_{h>0} \subset H^2(I;\R^3)$ and a limit $u\in H^2(I;\R^3)$
with $u_h \wto u$ in $H^2(I;\R^3)$ as $h\to 0$. To show that 
$E[u]\le \liminf_{h\to 0} E_{h,\veps}[u_h]$ it suffices to consider the case that the bound
is finite. Since $\cI_{1,h} |u_h'|^2 = 1$ we find that
\[
\big\| |u_h'|^2 - 1\big\|_{L^2(I)} \le c h \big\| \big(|u_h'|^2\big)'\|_{L^2(I)}
\le c h \|u_h''\|_{L^2(I)} \|u_h'\|_{L^\infty(I)},
\]
which by embedding results 
implies that $|u'|^2 =1$ in $I$. Similarly, since $\|\cI_{1,h} V_D(u_h)\|_{L^1(I)} \to 0$ 
it follows that $u\in D$ in $I$. Since the bending energy is weakly lower semicontinuous
and the potential term non-negative, we deduce the asserted inequality.  \\
(ii) Given $u\in \cA \cap D$  such that $u\in H^3(I;\R^3)$ we define $u_h = \cI_{3,h}u$
and note that $u_h \to u$ in $H^2(I;\R^3)$ and $u_h(x_i) \in D$ as well as
$|u_h'(x_i)|=1$ for all $i=0,1,\dots,N$, in particular $u_h\in \cA_h$. 
This implies that $\lim_{h \to 0} E_{h,\veps}[u_h] = E_\text{bend}[u]$. 
\end{proof}

\begin{remark}
The regularity condition can be avoided if a density result for inextensible
confined curves in the spirit of~\cite{Horn21} is available. Alternatively, 
a standard regularization of a given curve $u\in \cA\cap D$ can be considered
following~\cite{BarRei20,Bonito22} which requires an appropriate scaling of the 
discretization and penalty parameters.
\end{remark}

The proposition implies the convergence of discrete (almost) minimizers provided
that the boundary conditions imply a coercivity proper and exact minimizers are regular.

\begin{corollary}
Assume that minimizers $u\in \cA\cap D$ for $E_\text{bend}$ satisfy $u\in H^3(I;\R^3)$,
and that there exists $c>0$ such that $\|v\|_{H^2(I)} \le c \|v''\|_{L^2(I)}$ for 
all $v\in H^2(I;\R^3)$ with $L_\text{bc}[v]=0$ or $D$ is bounded. 
Then sequences of discrete almost minimizers
for $E_{h,\veps}$ accumulate weakly in $H^2(I;\R^3)$ at minimizers for $E_\text{bend}$.
\end{corollary}

Our second convergence result concerns an estimate on the confinement violation.

\begin{proposition}\label{prop:penetration}
Let $u_h \in \cA_h$. Then we have that
\[
\|\cI_{1,h} (|u_h|_D-1)_+\|_{L^\infty(I)} \le c \veps^{1/3} \big(E_{h,\veps}[u_h]\big)^{1/3}.
\]
\end{proposition}

\begin{proof}
The Gagliardo--Nirenberg inequality bounds the norm in $L^p(I)$ 
by the product of norms in $L^q(I)$ and $W^{1,r}(I)$ with exponents $1-\a$ and $\a$ 
such that $\frac1p+\frac1q=\alpha(1-\frac1r+\frac1q)$, cf.~\cite{Leon17-book}.
With $p=r=\infty,q=2,\alpha=1/3$ we have
\[
\|\cI_{1,h} (|u_h|_D-1)_+\|_{L^\infty(I)}\leq 
c \|\cI_{1,h} (|u_h|_D-1)_+\|_{L^2(I)}^{2/3}\|\cI_{1,h}(|u_h|_D-1)_+\|_{W^{1,\infty}(I)}^{1/3}.
\]
The $W^{1,\infty}$ norm can be uniformly bounded due to stability of the nodal
interpolation operator in $W^{1,\infty}$ and the nodal constraints $|u_h'(x_i)|^2 =1$, $i=0,1,\dots,N$. 
The term  $(\veps^{-1} \|\cI_{1,h}(|u_h|_D-1)_+\|_{L^2}^{2})^{1/3}$ is bounded by the third
root of the potential part of the discrete energy. 
\end{proof}

\begin{remark} 	\label{rem:epsilon-dependency}
A stronger estimate on the constraint violation can be derived if the solution
$u$ and the Lagrange multiplier $\l$ are sufficiently regular, so that the 
Euler--Lagrange equations hold in strong form, i.e., $\kappa u^{(4)} +  \veps^{-1} \nabla V_D(u) = (\l u')'$, 
which implies $\|\nabla V_D(u)\|_{L^\infty(I)} = \cO(\veps)$, where $|\nabla V_D(y|$ is proportional
to the distance of a point $y\in \R^3$ to the set $D$.
\end{remark} 

Our third convergence result follows from the unconditional energy stability of the
numerical scheme and states that the sequence of corrections $(d_t u_h^k)_{k=1,2,\dots}$ 
converges to zero as $k\to \infty$. Moreover, it provides a bound on the violation
of the arclength constraint due to its linearized treatment.

\begin{proposition}
The iterates $(u_h^k)_{k=0,1,\dots}$ of the scheme~\eqref{eq:fully_discr_grad_flow} satisfy
\[
E_{h,\veps}[u_h^K] + \tau \sum_{k=1}^K \|d_t u_h^{k+1}\|_\star^2 \le E_{h,\veps}[u_h^0],
\]
for all $K\ge 0$, and, provided that $\|v\|_{L^\infty(I)} \le c_\star \|v\|_\star$,
\[
\|\cI_{1,h} |u_h'|^2 - 1\|_{L^\infty(I)} \le c_\star^2 \tau E_{h,\veps}[u_h^0].
\]
\end{proposition}

\begin{proof}
By concavity of $V_D^\text{cv}$ we have that
\[
V_D^\text{cv}(u_h^k) + \nabla V_D^\text{cv}(u_h^k) \cdot \big( u_h^{k+1} - u_h^k\big) \ge V_D^\text{cv}(u_h^{k+1}) .
\]
This implies that by choosing $v_h = d_t u_h^{k+1}$ in~\eqref{eq:fully_discr_grad_flow} we have
\[
\|d_t u_h^{k+1}\|^2 + d_t \Big\{ \frac12 \|[u_h^{k+1}]''\|^2 
+ \frac{1}{2\veps} \int_I \cI_h V_D(u_h^{k+1}) \dv{x} \Big\} \le 0.
\]
Multiplication by $\tau$ and summation over $k=0,1,\dots,K-1$ yield the stability 
estimate. The orthogonality $[u_h^{k-1}]' \cdot [d_t u_h^k]'$ for $k=1,2,\dots,K$ 
at the nodes leads to the relation 
\[
|[u_h^k]'|^2  = |[u_h^{k-1}]'|^2 + \tau^2 |[d_t u_h^k]'|^2.
\]
Summing this identity over $k=1,2,\dots,K$, noting $|[u_h^0]'|^2 =1$ at the nodes,
and including the energy stability prove the estimate. 
\end{proof}

\section{Elasticae in balls and cylinders}
	
Our numerical calculations are performed in \textsc{Matlab} and with the \knotevolve web application 
\cite{BaFaWe20}. We consider closed elastic rods confined to balls and cylinders.
The highly symmetric stationary configurations found for balls give rise to the following 
definition that provides a concise classification via two integer numbers.  
For the scalar product $(d_tu_h^{k+1},v_h)_\star$, we use the $L^2$ inner product. The parameter
values $\kappa=10$ and $\veps=1/(10\kappa)$ are employed if not stated otherwise.

\begin{definition}
An arclength parametrized curve is called a \emph{$\mu$-circle} if it is a $\mu$-fold covered 
planar circle. It is called a \emph{$\mu$-$\nu$-clew} if it shows a $\nu$-fold symmetry around 
one axis running through the center of the ball. The integer $\mu$ is then defined as the winding number of $u$ around the 
rotational axis.
\end{definition}
	
As an illustrative example of the gradient flow, we use a trefoil knot of length $31.9$ 
that is confined to a ball of radius $4.6$ with $h\approx0.3$ and $\tau=0.1h$.~\footnote{The example can be run via
\url{aam.uni-freiburg.de/knotevolve/torus-2-3-97?Rho=0&CnfmType=ellipsoid&CnfmRadius=4.6,4.6,4.6&tmax=30000&StepW=0.1}}
Snapshots of the evolution are depicted in Figure~\ref{fig:evolution}. 
Also, the bending energy $\kappa\|[u_h^k]''\|^2/2$, the confinement energy $(2\veps)^{-1}E_D[u_h^k]$ and the violation of arc-length parametrization $\|\mathcal{I}_{1,h}\{|[u_h^k]'|^2-1\}\|_{L^\infty}$ are visualized as a function of $k$.

First, the trefoil 
knot evolves into a double-covered circle. At some point, it unfolds into a bent 
lemniscate whose outermost points reach the surface of the ball. This configuration then 
moves to the left and starts to unfold into a buckled circle that runs close to 
the ball's surface. The final elastica is a 1-2-clew. The symmetry axis of the elastica, which is also depicted in Figure~\ref{fig:evolution}, is different from the symmetry axis of the initial curve. 
The local curvature of the 1-2-clew is periodic along the curve with periodicity 4. We generally observe that the curvature of a $\mu$-$\nu$-clew is $2\nu$-periodic.

A similar shape 
was previously also obtained for modelling semiflexible biopolymers in spherical domains 
that are slightly smaller than the flat circle of the same length, cf.~\cite{Ostermeir2010}. 
The shape that we call 1-2-clew also arises when packing a thick rope of maximal length 
without self-penetration on the sphere \cite{Gerlach2011}. 
	
\begin{figure}
\centering
\includegraphics[width=\linewidth]{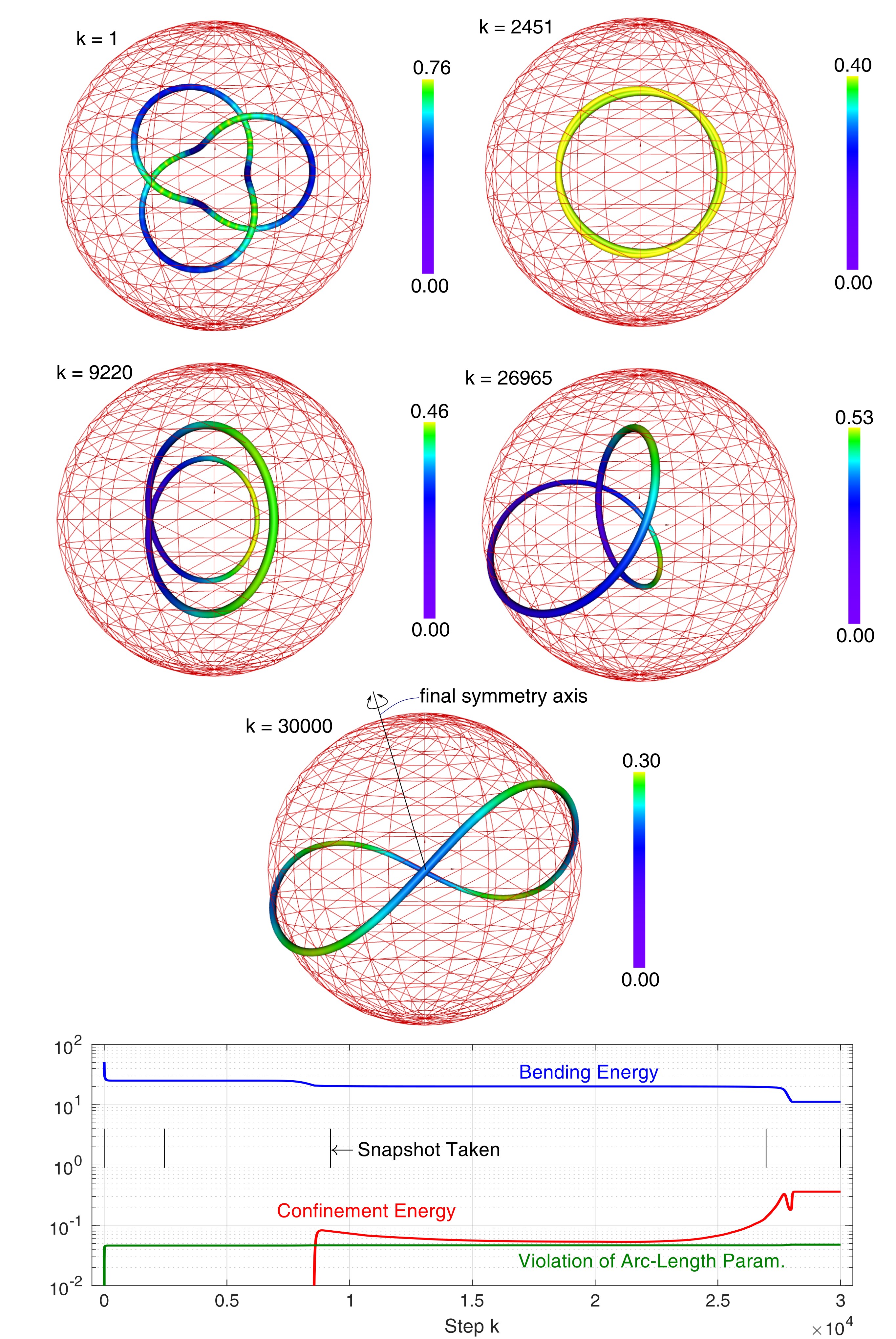}
\caption{Evolution of a closed elastic curve in a ball of radius $R=4.6$ to a
1-2-clew. The number
 $k$ indicates the time steps, the color scheme represents local curvature. Confinement and 
total energy, and violation of arc-length parametrization are shown in the lower 
right panel. }
\label{fig:evolution}
\end{figure}
	
When confining a rod of length $L$ to balls of varying radii, a multitude of equilibrium 
configurations is observed. Examples are illustrated in Figure~\ref{fig:clews} where we used
$C_4$ and $C_5$ symmetric initial configurations. Both, the symmetry number $\nu$ and the
winding number $\mu$ depend on the symmetry of the initial configuration and 
on the ratio $L/R$. All elasticae run close 
to the ball's surface and slightly exceed the confining domain. With decreasing radius, either 
the winding number or the symmetry number are gradually increased by 2. This follows 
from an increasing number of self-intersections of the rod that always affect two of 
its segments. The gradient flow used to minimize the energy preserves certain symmetries.
When starting with an even symmetry, the elastica is an odd-even-clew 
or an odd-circle as illustrated in Figure~\ref{fig:clews}. An initially odd rotational symmetry 
in turn generally leads to even-odd-clews or even-circles. An exception is the 
transition from the 2-1-clew to the 1-2-clew as illustrated in the introductory example. 
This transition involves large deformations. When the radius of the ball is too small, 
the confinement is too restrictive for the curve to undergo such large transitions.

\begin{figure}[p]
\centering
\includegraphics[width=.92\linewidth]{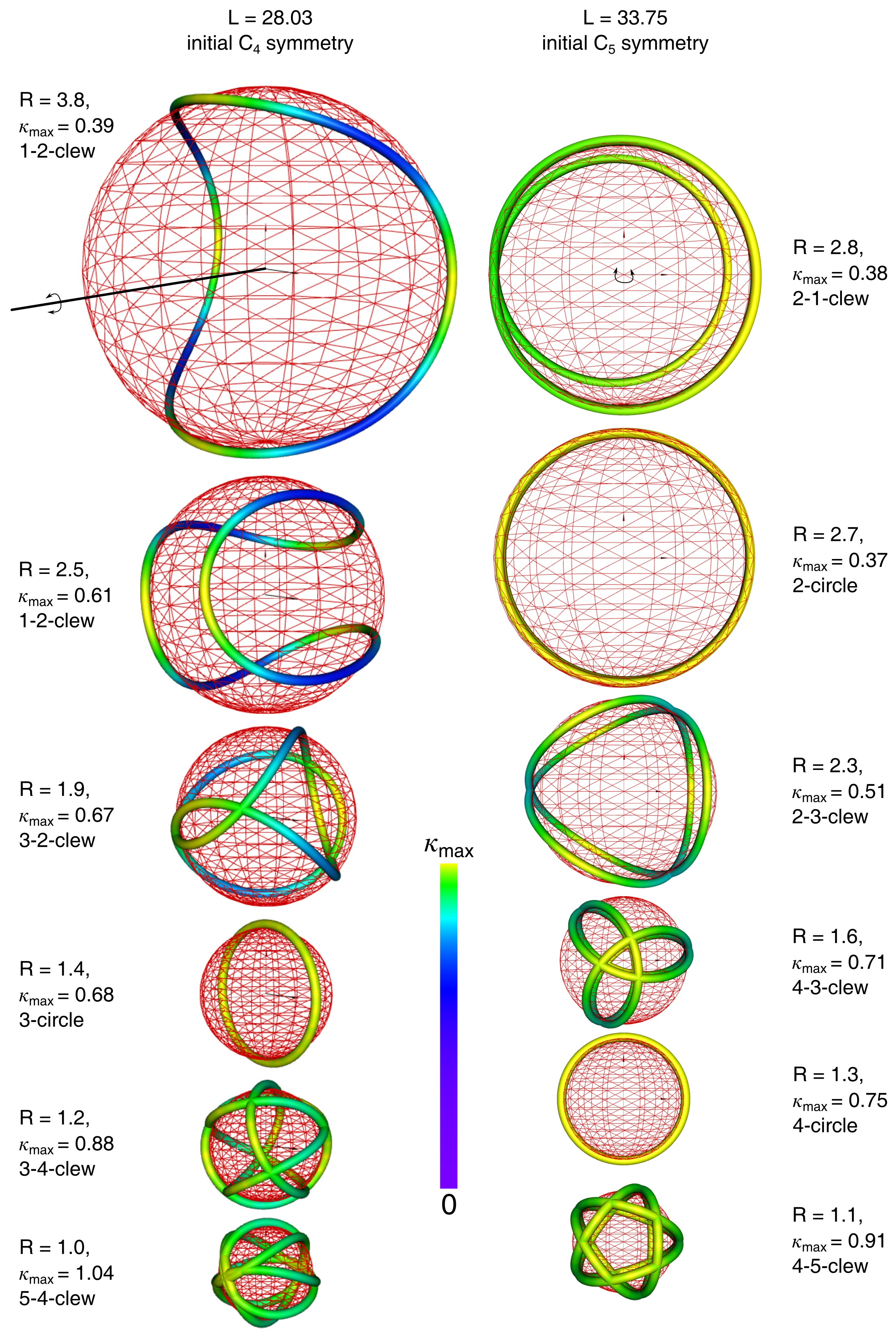}
\caption{Elastica shapes obtained for two different initial configurations (left and 
right column, respectively) and varying radii $R$. The color of the rod indicates the 
local curvature. The symmetry axis of the left column is indicated for the top configuration. 
In the right column, symmetry and view axes coincide.}
\label{fig:clews}
\end{figure}

The unconfined elastica of a closed rod of length $L$ is the circle of radius $r_L=L/(2\pi)$. The bending energy of this elastica is given by $E_{L}=\kappa L/(2r_L^2)=2\kappa\pi^2/L$. 
To categorize the equilibrium configurations of closed rods that are confined to balls of radius $R$, we evaluate $(E_\text{bend}/E_L)^{1/2}$ as a function of $r_L/R$. The first quantity measures the excess bending induced by the confinement, whereas the second quantity determines how many times too small the confining ball is compared to the unconfined elastica. We remark that for $\mu$-fold covered circles, both quantities equal $\mu$. 
	
When starting with the four- and five-fold symmetric initial configurations, we observe 
two distinct families of the bending energy dependency on the ratio $r_L/R$. The result 
is shown in Figure~\ref{fig:elasticae}. The elasticae in the five-fold symmetric case follow 
the pattern 1-circle, 1-2-clew, 3-2-clew, 3-circle, 3-4-clew, 5-4-clew, etc. 
for increasing $r_L/R$. In the four-fold symmetric case, we find (1-circle, 1-2-clew,) 
2-1-clew, 2-circle, 2-3-clew, 4-3-clew, 4-circle, 4-5-clew, and so on. The first two are 
special as they involve a large deformation of the rod when transiting from the even-odd 
to odd-even. Both patterns are very regular and are expected to continue for larger $r_L/R$. 
	
During the unfolding process, intermediate nearly stationary configurations are observed. 
These include a shape that could be called a 1-3-clew or multiply covered circles that 
are completely inside the ball. This can be seen in Figure~\ref{fig:evolution}: The initial 
rod evolves into a two-fold covered circle in the first place; later, the circle opens up. 
These configurations seem to be saddle-point structures as they are attractive with respect 
to the previous configuration, whereas there are adjacent configurations with smaller 
total energy. As the numerical representation of the rod cannot match those saddle-points 
perfectly, the rod exits those configurations after a certain number of steps. 
	
For irregularly shaped initial rod configurations, the previously described elastica shapes 
are found as well.~\footnote{See for instance a knot with crossing number 10 relaxing 
into a 3-2-clew: \url{aam.uni-freiburg.de/knotevolve/10_053?Rho=0&StepW=0.1&CnfmType=ellipsoid&CnfmRadius=3,3,3}}
Hence, the symmetry of the final shape can be attributed to solely the ratio $r_L/R$ 
and to the question whether the initial configuration prefers the odd-even or even-odd 
elasticae family.
	
\begin{figure}
\centering
\includegraphics[width=.8\linewidth]{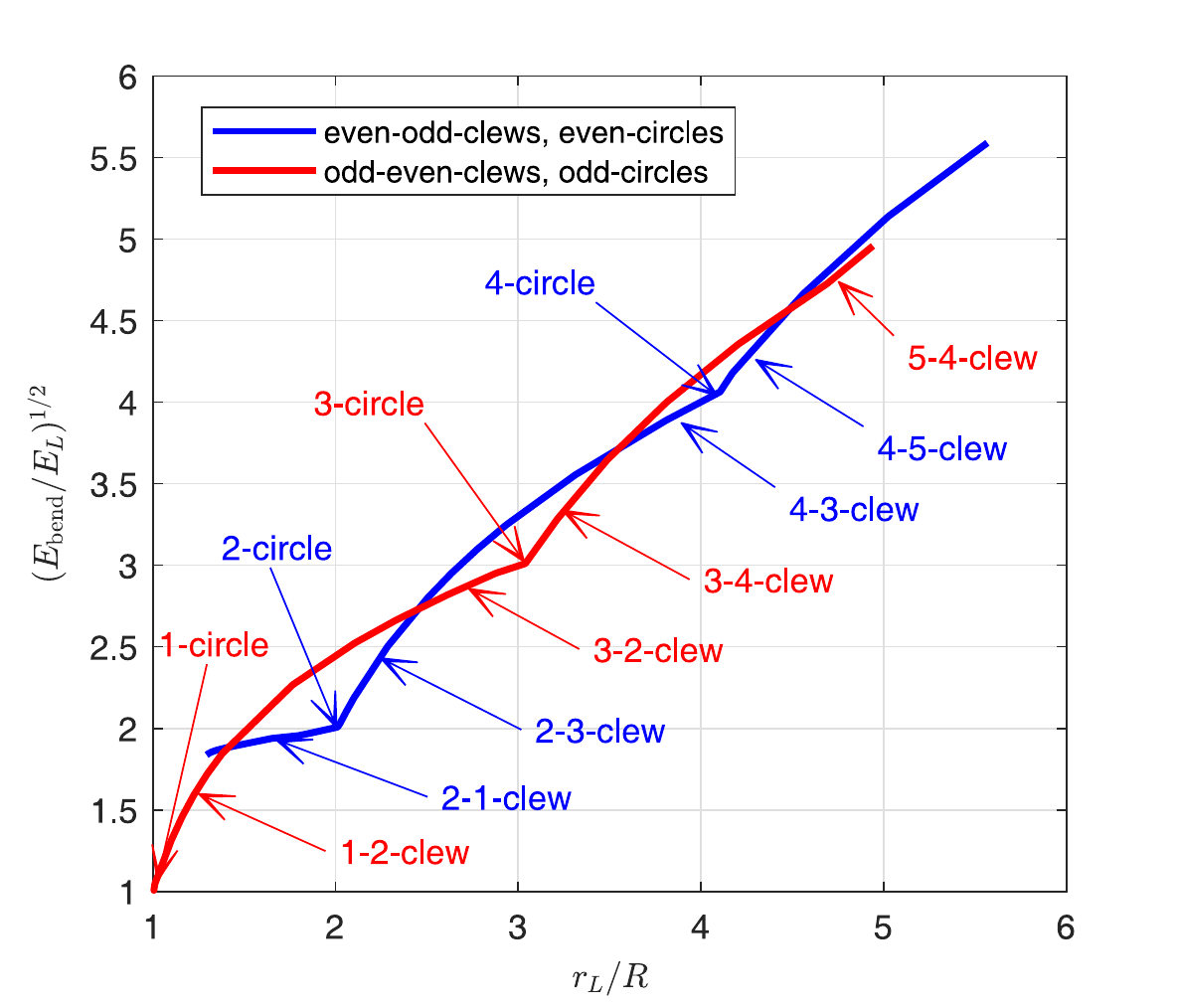}
\caption{Normalized square root of the bending energy depending on the ratio $r_L/R$ 
for the two observed elastica families}
\label{fig:elasticae}
\end{figure}

\begin{conjecture}
Consider a rod of length $L$ confined to a ball of radius $R$. Let $r_L=L/(2\pi)$ be the 
radius of the unconfined elastica and let $j\in\mathbf{N}$ be a 
positive integer with $j<r_L/R<j+1$. Then there is a number $\xi(j)\in(0,1)$ such 
that the globally least bent confined elastica is a $j$-($j+1$)-clew if $r_L/R<j+\xi$. 
Else, the global optimizer is the ($j+1$)-$j$-clew. If $r_L/R<1$, the flat 1-circle 
is an exactly confined elastica. In the cases where $r_L/R=j$ holds, the $j$-fold 
covered circle is the global elastica.
\end{conjecture}

The parameter $\varepsilon$ can be understood as a length-scale of maximal penetration 
into the complement of $D$. The numerical results shown in Figure~\ref{fig:epsilon}, indicate 
that
\[
\max\limits_{j=0,\ldots,N}(|u_h^k(x_j)|_D-1)_+=\mathcal{O}(\varepsilon).
\]
Thus, the maximal penetration decreases linearly with $\varepsilon$ whilst not depending on 
whether the initial rod configuration lies inside $D$. This indicates that the estimate as 
sketched in Remark~\ref{rem:epsilon-dependency} is valid in the case of spherical confinements. 

It can be observed that when varying $\varepsilon$ for the 
same initial configuration the final shape is not unique (cf. Figure~\ref{fig:epsilon}): 
In the first example (dots in Figure~\ref{fig:epsilon}), mainly 1-2-clews were obtained as 
final shapes, but for one value of $\veps$, also a 2-1-clew could be observed. 
The second example that relaxed to a 2-3-clew for sufficiently small $\varepsilon$ 
(crosses in Figure~\ref{fig:epsilon}), turned 
into a two-fold covered circle if the confinement was too weak.

\begin{figure}
	\centering
	\includegraphics[width=.6\linewidth]{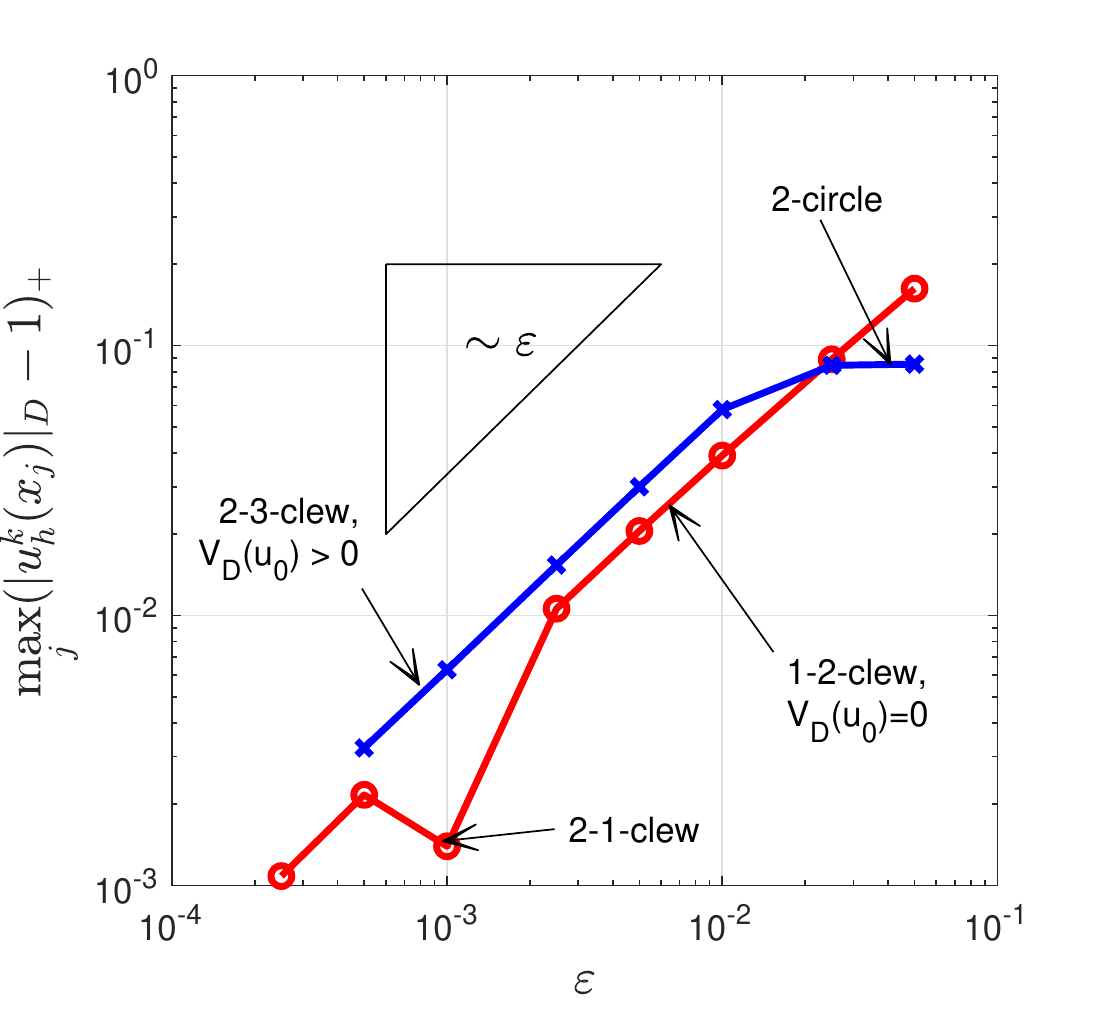}
	\caption{Maximal nodal penetration outside the confinement 
		domain depending on $\varepsilon$.}
	\label{fig:epsilon}
\end{figure}
	
Interestingly, all elasticae for $r_L/R\geq1$ lie on the surface of the ball up to the 
penetration due to the finite potential. Brunnett and Crouch \cite{Brunnett1994} derived a 
differential equation for the geodesic curvature $\kappa_g$ of the rod on the sphere, i.e. 
the projection of the total curvature on the tangent plane in each point: 
\begin{equation*}
\kappa_g''+\frac12\kappa_g^3+C\kappa_g=0.
\end{equation*}
Here, $C$ is a constant consisting of the tension energy of the rod and 
the sphere square curvature. This differential equation can be solved by the Jacobi elliptic 
cosine function. The parameters of these functions must be adjusted such that the curvature is periodic. When using the solution to calculate the actual rod position, a nine-component ODE is solved, again imposing periodicitiy.
This raises the question whether the resulting configurations coincide with our experimentally observed $\mu$-$\nu$-clews and $\mu$-circles, thus leading to an analytic definition of our clews. For the $\mu$-$\nu$-clew, $\nu$ and $\mu$ should arise when taking the periodic boundary conditions into account for the geodesic curvature and the rod position, respectively. It would be also of analytic interest if all elasticae confined to balls are actually elasticae on the sphere.
A proof would however go beyond the scope of this manuscript.

We close our discussion by experimentally investigating closed curves confined to cylinders
of different heights and radii.~\footnote{See for instance \url{aam.uni-freiburg.de/knotevolve/torus-1-13-100?Rho=0&StepW=0.2&CnfmType=cylinder-z&CnfmRadius=3,3,3&tmax=200000}} As shown in Figure~\ref{fig:elasticae_cyl}, large cylinder heights apparently lead to flat configurations that resemble semi-circles connected by straight lines. If the cylinder is long enough, we assume that the true global elastica consists of two semi-circles connected by two straight lines. If height and diameter are equal, we observe a shape similar to a 4-3-clew (for height and diameter 6) or a 3-2-clew (for height and diameter 8). Other combinations of height and diameter reveal a large variety of optimal shapes. A concise classification as in the case of spherical confinement however is not obvious.

\begin{figure}[p]
\begin{minipage}{.45\linewidth}
\includegraphics[width=1\linewidth]{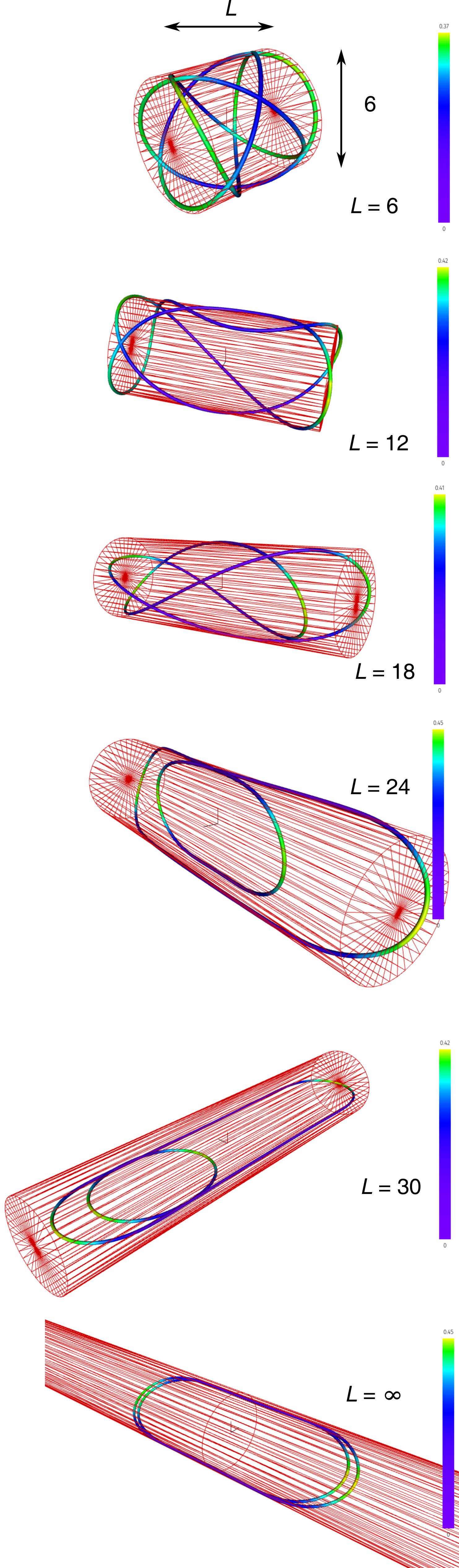}
\end{minipage}
\hspace*{5mm}
\begin{minipage}{.42\linewidth}
\vspace*{-66mm}
\includegraphics[width=1\linewidth]{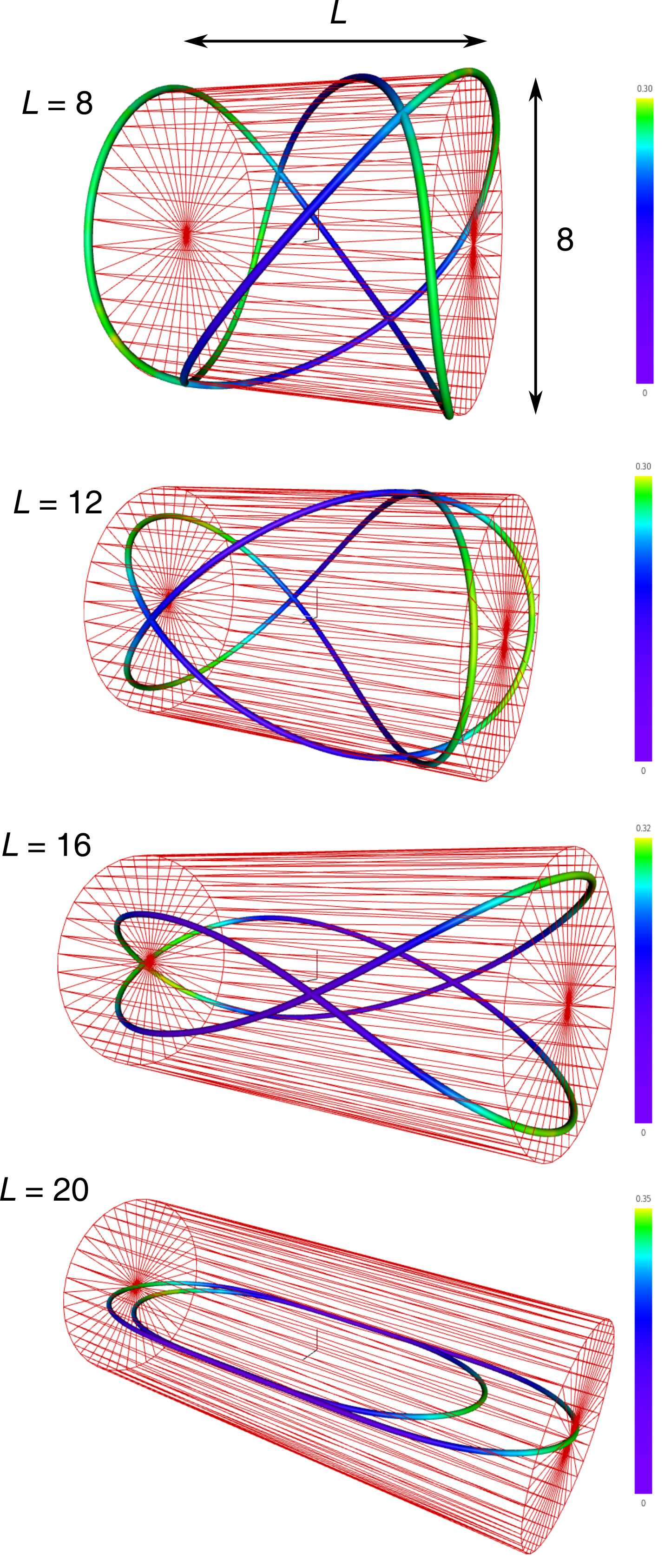}
\end{minipage}
\caption{Stationary closed curves in cylindrical domains with radii
$R=3$ (left) and $R=4$ (right) and different heights (increasing
from top to bottom).}
\label{fig:elasticae_cyl}
\end{figure}
 
\section*{References}
\printbibliography[heading=none]

\end{document}